\documentclass{amsart}
\usepackage[dvips]{graphicx}
\usepackage{amssymb}

\newtheorem{theorem}{Theorem}

\newtheorem{lemma}[theorem]{Lemma}
\newtheorem{proposition}[theorem]{Proposition}

\newtheorem{claim}[theorem]{Claim}

\theoremstyle{remark}
\newtheorem{definition}[theorem]{\sc Definition}

\newtheorem{example}[theorem]{\sc Example}
\newtheorem*{acknowledgment}{\sc Acknowledgment}

\numberwithin{equation}{section}

\begin{document}

\title[Classification of toric log del Pezzo surfaces]
{Classification of toric log del Pezzo surfaces with few singular points}

\author{Yusuke Suyama}
\address{Department of Mathematics, Graduate School of Science, Osaka University,
1-1 Machikaneyama-cho, Toyonaka, Osaka 560-0043 JAPAN}
%\curraddr{}
\email{y-suyama@cr.math.sci.osaka-u.ac.jp}
%\thanks{}



\date{\today}

%\dedicatory{}

\begin{abstract}
We give a classification of toric log del Pezzo surfaces with two or three singular points.
\end{abstract}

\maketitle

\section{Introduction}

A normal projective surface is called a {\it log del Pezzo surface}
if it has at worst log-terminal singularities (that is, quotient singularities)
and its anticanonical divisor is a $\mathbb{Q}$-Cartier ample divisor.
Log del Pezzo surfaces have been extensively studied
and many results are known (for example \cite{N1, N2, N3, AN, N, FY}).

An $n$-dimensional {\it toric variety} is a normal variety $X$ over $\mathbb{C}$
containing the algebraic torus $(\mathbb{C}^*)^n$ as an open dense subset,
such that the natural action of $(\mathbb{C}^*)^n$ on itself extends to an action on $X$.
There is a one-to-one correspondence between toric log del Pezzo surfaces
and certain lattice polygons, called {\it LDP-polygons} (see Section 2).
LDP-polygons were introduced by Dais--Nill \cite{DN}
and the classification of toric log del Pezzo surfaces
was also studied by several researchers.
Dais \cite{D07, D09} classified toric log del Pezzo surfaces
of Picard number one and index at most three,
and Kasprzyk--Kreuzer--Nill \cite{KKN} gave two independent algorithms
that classify all toric log del Pezzo surfaces.

In this paper, we focus on the number of singular points on toric log del Pezzo surfaces.
It is well known that there are exactly five nonsingular toric del Pezzo surfaces.
On the other hand,
in the general case, Belousov \cite{B1, B2} proved that a log del Pezzo surface of Picard number one
has at most four singular points,
and Kojima \cite{K} classified log del Pezzo surfaces of Picard number one
with unique singular points.
Recently, Dais classified all toric log del Pezzo surfaces with unique singular points:

\begin{theorem}[{\cite[Theorem 1.4]{D17}}]\label{Dais}
Let $Q$ be an $LDP$-polygon.
Then the associated toric log del Pezzo surface $X_Q$ has exactly one singular point
if and only if there exists a positive integer $p$
such that $Q$ is equivalent to one of the following:
\begin{enumerate}
\item $\mathrm{conv}\left\{\left(\begin{array}{c} 1 \\ -1 \end{array}\right),
\left(\begin{array}{c} p \\ 1 \end{array}\right),
\left(\begin{array}{c} -1 \\ 0 \end{array}\right)\right\}$.
\item $\mathrm{conv}\left\{\left(\begin{array}{c} 1 \\ -1 \end{array}\right),
\left(\begin{array}{c} p \\ 1 \end{array}\right),
\left(\begin{array}{c} p-1 \\ 1 \end{array}\right),
\left(\begin{array}{c} -1 \\ 0 \end{array}\right)\right\}$.
\item $\mathrm{conv}\left\{\left(\begin{array}{c} 1 \\ -1 \end{array}\right),
\left(\begin{array}{c} p \\ 1 \end{array}\right),
\left(\begin{array}{c} p-1 \\ 1 \end{array}\right),
\left(\begin{array}{c} -1 \\ 0 \end{array}\right),
\left(\begin{array}{c} 0 \\ -1 \end{array}\right)\right\}$.
\end{enumerate}
\end{theorem}

The purpose of this paper is to extend this classification
for those with two or three singular points:

\begin{theorem}\label{2singularities}
Let $Q$ be an $LDP$-polygon.
Then the associated toric log del Pezzo surface $X_Q$ has exactly two singular points
if and only if one of the following conditions is satisfied:
\begin{enumerate}
\item There exist $p, q \in \mathbb{Z}$
such that $p, q \geq 2, \gcd(p, q)=1$ and
$Q$ is equivalent to
\begin{equation*}
\mathrm{conv}\left\{\left(\begin{array}{c} 1 \\ 0 \end{array}\right),
\left(\begin{array}{c} 0 \\ 1 \end{array}\right),
\left(\begin{array}{c} -p \\ -q \end{array}\right)\right\}.
\end{equation*}
\item There exist $p, q, r \in \mathbb{Z}$ such that
$p \leq 1, r \leq \min\{-pq-2, -2, -q-1, q-pq-1\}, \gcd(q, r)=1$ and
$Q$ is equivalent to
\begin{equation*}
\mathrm{conv}\left\{\left(\begin{array}{c} 1 \\ 0 \end{array}\right),
\left(\begin{array}{c} 0 \\ 1 \end{array}\right),
\left(\begin{array}{c} -1 \\ p \end{array}\right),
\left(\begin{array}{c} q \\ r \end{array}\right)\right\}.
\end{equation*}
\item There exist $p, q, r \in \mathbb{Z}$ such that
$p \leq 0, 1 \leq q \leq -r-1, \gcd(q, r)=1$ and
$Q$ is equivalent to
\begin{equation*}
\mathrm{conv}\left\{\left(\begin{array}{c} 1 \\ 0 \end{array}\right),
\left(\begin{array}{c} 0 \\ 1 \end{array}\right),
\left(\begin{array}{c} -1 \\ p+1 \end{array}\right),
\left(\begin{array}{c} -1 \\ p \end{array}\right),
\left(\begin{array}{c} q \\ r \end{array}\right)\right\}.
\end{equation*}
\end{enumerate}
In particular, if $X_Q$ has exactly two singular points,
then the Picard number of $X_Q$ is at most three.
\end{theorem}

\begin{theorem}\label{3singularities}
Let $Q$ be an $LDP$-polygon.
Then the associated toric log del Pezzo surface $X_Q$ has exactly three singular points
if and only if one of the following conditions is satisfied:
\begin{enumerate}
\item The Picard number of $X_Q$ is at most two.
\item There exist $p, q, r, s, t \in \mathbb{Z}$ such that
\begin{align*}
p & \leq 1,\\
r & \leq \min\{-1, -pq-2, q-pq-1, -pq+qt-rs+ps+t-1\},\\
t & \leq \min\{-2, -s-1, qt-rs+r-1\},\\
2 & \leq qt-rs,
\end{align*}
$\gcd(q, r)=\gcd(s, t)=1$ and $Q$ is equivalent to
\begin{equation*}
\mathrm{conv}\left\{\left(\begin{array}{c} 1 \\ 0 \end{array}\right),
\left(\begin{array}{c} 0 \\ 1 \end{array}\right),
\left(\begin{array}{c} -1 \\ p \end{array}\right),
\left(\begin{array}{c} q \\ r \end{array}\right),
\left(\begin{array}{c} s \\ t \end{array}\right)\right\}.
\end{equation*}
\item $X_Q$ is isomorphic to the blow-up of a toric log del Pezzo surface of Picard number three
at one nonsingular $(\mathbb{C}^*)^2$-fixed point.
\end{enumerate}
In particular, if $X_Q$ has exactly three singular points,
then the Picard number of $X_Q$ is at most four.
\end{theorem}

The structure of the paper is as follows:
In Section 2, we collect basic results from toric geometry.
In Sections 3 and 4,
we prove Theorems \ref{2singularities} and \ref{3singularities}, respectively.

\begin{acknowledgment}
This work was supported by JSPS KAKENHI Grant Number JP18J00022.
\end{acknowledgment}

\section{Toric log del Pezzo surfaces}

We fix a notation and recall some basic facts from toric geometry which will be used in this paper,
see \cite{CLS} for details.
Let $\Delta$ be a complete fan in $\mathbb{R}^2$.
We list the primitive generators of one-dimensional cones in $\Delta$
as $v_1, \ldots, v_d$ in counterclockwise order around the origin in $\mathbb{R}^2$,
and we define $v_0=v_d$ and $v_{d+1}=v_1$.
For $i=1, \ldots, d$, we write $v_i=\left(\begin{array}{c} x_i \\ y_i \end{array}\right)$.
We denote by $X(\Delta)$
the associated complete toric surface.
Since $\Delta$ is simplicial, the Picard number $\rho(X(\Delta))$ of $X(\Delta)$ equals $d-2$.

For $i=1, \ldots, d$, we denote by $D_i$ the torus-invariant divisor
on $X(\Delta)$ corresponding to $v_i$,
and we define $\sigma_i=\mathbb{R}_{\geq 0}v_i+\mathbb{R}_{\geq 0}v_{i+1}$.
For a two-dimensional cone $\sigma$ in $\Delta$,
we denote by $\mathrm{orb}(\sigma)$ the corresponding $(\mathbb{C}^*)^2$-fixed point
of $X(\Delta)$.
Then $\mathrm{orb}(\sigma_i)$ is nonsingular if and only if
$\det(v_i, v_{i+1})=1$.
The set $\mathrm{Sing}(X(\Delta))$ of singular points of $X(\Delta)$
is $\{\mathrm{orb}(\sigma_i); 1 \leq i \leq d, \det(v_i, v_{i+1}) \geq 2\}$.

\begin{definition}
Let $\Delta$ be a complete fan in $\mathbb{R}^2$.
We define the map $f:\{1, \ldots, d\} \rightarrow \mathbb{Z}$
by $f(i)=\det(v_{i-1}, v_i)+\det(v_i, v_{i+1})+\det(v_{i+1}, v_{i-1})$.
\end{definition}

\begin{proposition}\label{intersection number}
Let $\Delta$ be a complete fan in $\mathbb{R}^2$.
Then the toric surface $X(\Delta)$ is log del Pezzo if and only if
$f(i) \geq 1$ for every $i=1, \ldots, d$.
\end{proposition}

\begin{proof}
The toric surface $X(\Delta)$ is log del Pezzo if and only if the intersection number
$(-K_{X(\Delta)} \cdot D_i)$ of the anticanonical divisor $-K_{X(\Delta)}$ with $D_i$
is positive for every $i=1, \ldots, d$ (see \cite[Theorem 6.3.13]{CLS}).
We see that
\begin{equation*}
\frac{1}{\det(v_{i-1}, v_i)}v_{i-1}+\frac{1}{\det(v_i, v_{i+1})}v_{i+1}
+\frac{\det(v_{i+1}, v_{i-1})}{\det(v_{i-1}, v_i)\det(v_i, v_{i+1})}v_i=0
\end{equation*}
for every $i=1, \ldots, d$.
Hence
\begin{align*}
(-K_{X(\Delta)} \cdot D_i)&=\frac{1}{\det(v_{i-1}, v_i)}+\frac{1}{\det(v_i, v_{i+1})}
+\frac{\det(v_{i+1}, v_{i-1})}{\det(v_{i-1}, v_i)\det(v_i, v_{i+1})}\\
&=\frac{f(i)}{\det(v_{i-1}, v_i)\det(v_i, v_{i+1})}
\end{align*}
(see \cite[Theorem 8.2.3 and Proposition 6.4.4]{CLS}).
Since $\det(v_{i-1}, v_i)\det(v_i, v_{i+1})>0$, the assertion holds.
\end{proof}

A convex lattice polygon $Q \subset \mathbb{R}^2$ is called an {\it LDP-polygon}
if it contains the origin in its interior and every vertex is primitive.
Two LDP-polygons $Q$ and $Q'$ are said to be {\it equivalent}
if there exists a $2 \times 2$ unimodular matrix that transforms $Q$ into $Q'$.
For an edge $F$ of $Q$, we define a rational strongly convex polyhedral cone
$\sigma_F=\{\lambda x \mid \lambda \geq 0, x \in F\} \subset \mathbb{R}^2$.
Then the set $\Delta_Q$ of all such cones and their faces
forms a complete fan in $\mathbb{R}^2$.
We define $X_Q$ to be the associated toric surface $X(\Delta_Q)$.

\begin{proposition}[{\cite[Proposition 4.2]{D17}}]\label{LDP-polygon}
The correspondence $Q \mapsto X_Q$
induces a bijection between
equivalence classes of LDP-polygons
and isomorphism classes of toric log del Pezzo surfaces.
\end{proposition}

\begin{example}
Let $p$ be a positive integer and let
\begin{equation*}
v_1=\left(\begin{array}{c} 1 \\ 0 \end{array}\right),
v_2=\left(\begin{array}{c} 0 \\ 1 \end{array}\right),
v_3=\left(\begin{array}{c} -1 \\ -p \end{array}\right).
\end{equation*}
Then the convex hull $Q=\mathrm{conv}\{v_1, v_2, v_3\}$ is an LDP-polygon.
The fan $\Delta_Q$ consists of the cones $\mathbb{R}_{\geq 0}v_1+\mathbb{R}_{\geq 0}v_2,
\mathbb{R}_{\geq 0}v_2+\mathbb{R}_{\geq 0}v_3,
\mathbb{R}_{\geq 0}v_3+\mathbb{R}_{\geq 0}v_1$ and their faces.
The associated toric log del Pezzo surface $X(\Delta_Q)$
is the weighted projective plane $\mathbb{P}(1, 1, p)$
(see \cite[Example 3.1.17]{CLS}).
\end{example}

\section{Proof of Theorem \ref{2singularities}}

In this section, we give a proof of Theorem \ref{2singularities}.
We will use the notation in Section 2 freely.
The following lemmas play key roles in the proof of Theorems \ref{2singularities} and \ref{3singularities}.

\begin{lemma}\label{SRS}
Let $X(\Delta)$ be a toric log del Pezzo surface with $d \geq 4$.
Suppose that there exists $i$ with $1 \leq i \leq d$
such that $\sigma_{i-1}$ and $\sigma_{i+1}$ are singular while $\sigma_i$ is nonsingular.
Then the following hold:
\begin{enumerate}
\item We have $\det(v_{i+2}, v_{i-1}) \geq 2$.
In particular, the cones $\sigma_{i-1}, \sigma_i, \sigma_{i+1}$ cover more than a half-plane.
\item There exists $j \in \{1, \ldots, d\} \setminus \{i-1, i ,i+1\}$
such that $\sigma_j$ is singular.
In particular, $|\mathrm{Sing}(X(\Delta))| \geq 3$.
\end{enumerate}
\end{lemma}

\begin{proof}
Without loss of generality,
we may assume that $i=2, v_2=\left(\begin{array}{c} -1 \\ 0 \end{array}\right),
v_3=\left(\begin{array}{c} 0 \\ -1 \end{array}\right)$.
Then $y_1 \geq 2$ and $x_4 \geq 2$.

(1) Proposition \ref{intersection number} and $f(2)=x_1+y_1+1$ imply $x_1 \geq -y_1$.
Since $y_1 \geq 2$ and $v_1=\left(\begin{array}{c} x_1 \\ y_1 \end{array}\right)$ is primitive,
we have $x_1 \geq 1-y_1$ and $x_1 \ne 0$.
A similar argument shows that $y_4 \geq 1-x_4$ and $y_4 \ne 0$.

{\it Case 1}. Suppose that $x_1 \leq -1$ or $y_4 \leq -1$.
We may assume that $x_1 \leq -1$.
Then $y_4 \geq 1-x_4$ implies $x_1y_4 \leq x_1(1-x_4)$.
Since $1-x_4 \leq 1-2=-1$ and $x_1 \geq 1-y_1$,
we have $x_1(1-x_4) \leq (1-x_4)(1-y_1)$.
Thus $x_1y_4 \leq (1-x_4)(1-y_1)$.
Therefore $\det(v_4, v_1)=x_4y_1-x_1y_4 \geq x_4y_1-(1-x_4)(1-y_1)=x_4+y_1-1 \geq 3$.

{\it Case 2}. Suppose that $x_1 \geq 1$ and $y_4 \geq 1$.
Then both $v_4$ and $v_1$ are in
$\left\{\left(\begin{array}{c} x \\ y \end{array}\right) \in \mathbb{R}^2; x, y \geq 0\right\}$.
Since $v_1, \ldots, v_d$ are arranged in counterclockwise order,
we must have $\det(v_4, v_1)=x_4y_1-x_1y_4 \geq 1$.
Assume $\det(v_4, v_1)=1$ for contradiction.
Let $Q=\mathrm{conv}\{v_1, \ldots, v_d\}$.
We see that $(x_4+1)/x_1>0$ and
\begin{align*}
&\frac{x_4+1}{x_1}\left(\left(1-\frac{x_4}{x_4+1}\right)v_3+\frac{x_4}{x_4+1}v_1\right) \\
&=\frac{1}{x_1}\left(\begin{array}{c} x_1x_4 \\ -1+x_4y_1 \end{array}\right)
=\frac{1}{x_1}\left(\begin{array}{c} x_1x_4 \\ x_1y_4 \end{array}\right)=v_4.
\end{align*}
Since $v_4$ is a vertex of $Q$,
we must have $(x_4+1)/x_1>1$, so $x_4 \geq x_1$.
The assumption $\det(v_4, v_1)=1$ and $x_4 \geq 2$ imply $x_4-1 \geq x_1$.
A similar argument shows that $y_1-1 \geq y_4$.
Thus $x_1y_4 \leq (x_4-1)(y_1-1)$.
Therefore $1=\det(v_4, v_1)=x_4y_1-x_1y_4 \geq x_4y_1-(x_4-1)(y_1-1)=x_4+y_1-1 \geq 3$,
which is a contradiction.

In every case, we obtain $\det(v_4, v_1) \geq 2$.

(2) If $d=4$, then $\sigma_4$ is a singular cone by (1). Assume $d \geq 5$.
Then by (1), there exists a lattice point
$v \in (\mathrm{conv}\{0, v_4, v_1\} \cap \mathbb{Z}^2) \setminus \{0, v_4, v_1\}$.
Since
\begin{equation*}
\bigcup_{j=4}^d \mathrm{conv}\{0, v_j, v_{j+1}\}
\supsetneq \mathrm{conv}\{0, v_4, v_1\},
\end{equation*}
there exists $j \in \{4, \ldots, d\}$ such that
$v$ is an interior point of $\mathrm{conv}\{0, v_j, v_{j+1}\}$.
In particular, $\sigma_j$ is a singular cone.
\end{proof}

\begin{lemma}\label{RR}
Let $X(\Delta)$ be a singular toric log del Pezzo surface.
Suppose that $X(\Delta)$ cannot be obtained by blowing up
a toric log del Pezzo surface at a nonsingular $(\mathbb{C}^*)^2$-fixed point,
and there exists $i$ with $1 \leq i \leq d$ such that $\sigma_i$ and $\sigma_{i+1}$ are nonsingular.
Then $\sigma_j$ is singular for every $j \in \{1, \ldots, d\} \setminus \{i, i+1\}$.
\end{lemma}

\begin{proof}
It is obvious for $d=3$.
We may assume that $d \geq 4, i=2$ and
\begin{equation*}
v_2=\left(\begin{array}{c} -1 \\ 0 \end{array}\right),
v_3=\left(\begin{array}{c} 0 \\ -1 \end{array}\right),
v_4=\left(\begin{array}{c} 1 \\ a \end{array}\right)
\end{equation*}
for $a \in \mathbb{Z}$.
Assume for contradiction that
there exists $j \in \{1, 4, 5, \ldots, d\}$ such that $\sigma_j$ is nonsingular.
Proposition \ref{intersection number} and $f(3)=a+2$ imply $a \geq -1$.
Since $X(\Delta)$ cannot be obtained by blowing up a toric log del Pezzo surface
at a nonsingular $(\mathbb{C}^*)^2$-fixed point,
we have $v_2+v_4 \neq v_3$ (see, for example \cite[Proposition 3.3.15]{CLS}).
Thus $a \geq 0$ and $y_5, \ldots, y_d, y_1 \geq 1$.
In particular,
$\sigma_2 \cup \sigma_3
\supset \left\{\left(\begin{array}{c} x \\ y \end{array}\right) \in \mathbb{R}^2; y \leq 0\right\}$.

{\it Case 1}. Suppose that $5 \leq j \leq d$.
If $\sigma_{j-1}$ (resp.\ $\sigma_{j+1}$) is nonsingular,
then $a=0$ and
$\sigma_{j-1} \cup \sigma_j$ (resp.\ $\sigma_j \cup \sigma_{j+1}$)
coincides with $\left\{\left(\begin{array}{c} x \\ y \end{array}\right) \in \mathbb{R}^2; y \geq 0\right\}$,
a contradiction.
Hence both $\sigma_{j-1}$ and $\sigma_{j+1}$ are singular.
However, by Lemma \ref{SRS} (1),
the cones $\sigma_{j-1}, \sigma_j, \sigma_{j+1}$ cover more than a half-plane.
This is a contradiction.

{\it Case 2}. Suppose that $j=1$ or $j=4$.
We may assume that $j=4$.
Since $\Delta$ is singular,
we may further assume that $\sigma_1$ is singular.
Then $y_1 \geq 2$ and $v_5=\left(\begin{array}{c} b \\ ab+1 \end{array}\right)$
for some $b \in \mathbb{Z}$.
Note that $ab=y_5-1 \geq 0$.
Proposition \ref{intersection number} and $f(2)=x_1+y_1+1$ imply $x_1 \geq -y_1$.
Since $y_1 \geq 2$ and $v_1$ is primitive, we have $x_1 \geq 1-y_1$ and $x_1 \ne 0$.
Furthermore, $1 \leq f(4)=2-b$ and $v_3+v_5 \neq v_4$ imply $b \leq 0$.
There are three subcases to consider.

{\it Subcase 2.1}. Suppose that $x_1 \geq 1$.
Then $d \geq 5$ since $x_5=b \leq 0$.
However, $1 \leq \det(v_5, v_1)=by_1-(ab+1)x_1<0$, which is a contradiction.

{\it Subcase 2.2}. Suppose that $x_1 \leq -1$ and $a=0$. Then
$v_4=\left(\begin{array}{c} 1 \\ 0 \end{array}\right),
v_5=\left(\begin{array}{c} b \\ 1 \end{array}\right)$ and $d \geq 5$.
Let $Q=\mathrm{conv}\{v_1, \ldots, v_d\}$.
We see that $(\det(v_5, v_1)+1)/y_1>0$ and
\begin{equation*}
\frac{\det(v_5, v_1)+1}{y_1}\left(\left(1-\frac{1}{\det(v_5, v_1)+1}\right)v_4
+\frac{1}{\det(v_5, v_1)+1}v_1\right)=v_5.
\end{equation*}
However, $\det(v_5, v_1)+1=by_1-x_1+1 \leq by_1+y_1 \leq y_1$
and thus $0<(\det(v_5, v_1)+1)/y_1 \leq 1$,
which contradicts that $v_5$ is a vertex of $Q$.

{\it Subcase 2.3}. Suppose that $x_1 \leq -1$ and $a \geq 1$.
Since $ab \geq 0$ and $b \leq 0$, we must have $b=0$.
Hence
$v_4=\left(\begin{array}{c} 1 \\ a \end{array}\right),
v_5=\left(\begin{array}{c} 0 \\ 1 \end{array}\right)$ and $d \geq 5$.
Let $Q=\mathrm{conv}\{v_1, \ldots, v_d\}$.
We see that $(1-x_1)/\det(v_4, v_1)>0$ and
\begin{equation*}
\frac{1-x_1}{\det(v_4, v_1)}\left(\left(1-\frac{1}{1-x_1}\right)v_4
+\frac{1}{1-x_1}v_1\right)=v_5.
\end{equation*}
However, $\det(v_4, v_1)=y_1-ax_1 \geq (1-x_1)-ax_1 \geq 1-x_1+a>1-x_1$
and thus $0<(1-x_1)/\det(v_4, v_1)<1$,
which contradicts that $v_5$ is a vertex of $Q$.

Thus we have reached a contradiction in every case.
Hence $\sigma_1, \sigma_4, \sigma_5, \ldots, \sigma_d$ are singular.
\end{proof}

We are now ready to prove Theorem \ref{2singularities}.

\begin{proof}[Proof of Theorem \ref{2singularities}]
Let $X(\Delta)$ be a complete toric surface with two singular points.
Then there exists at least one $i$ with $1 \leq i \leq d$ such that $\sigma_i$ is nonsingular.
Hence we may assume that
$i=1, v_1=\left(\begin{array}{c} 1 \\ 0 \end{array}\right),
v_2=\left(\begin{array}{c} 0 \\ 1 \end{array}\right)$.

{\it The case where $d=3$}.
Every complete toric surface of Picard number one is log del Pezzo (see \cite[Lemma 6.4]{D09}).
Since $|\mathrm{Sing}(X(\Delta))|=2$, the cones $\sigma_2$ and $\sigma_3$ are singular.
Hence $v_3=\left(\begin{array}{c} -p \\ -q \end{array}\right)$
for some $p, q \in \mathbb{Z}$ with $p, q \geq 2$ and $\gcd(p, q)=1$.
Conversely,
for any $p, q \in \mathbb{Z}$ with $p, q \geq 2$ and $\gcd(p, q)=1$,
the convex hull
$\mathrm{conv}\left\{\left(\begin{array}{c} 1 \\ 0 \end{array}\right),
\left(\begin{array}{c} 0 \\ 1 \end{array}\right),
\left(\begin{array}{c} -p \\ -q \end{array}\right)\right\}$
is an LDP-polygon and the associated toric log del Pezzo surface
has exactly two singular points.

{\it The case where $d=4$}.
Suppose that $X(\Delta)$ is log del Pezzo.
By Lemma \ref{SRS} (2), either $\sigma_2$ or $\sigma_4$ is nonsingular.
We may assume that $\sigma_2$ is nonsingular.
Then $v_3=\left(\begin{array}{c} -1 \\ p \end{array}\right),
v_4=\left(\begin{array}{c} q \\ r \end{array}\right)$
for some $p, q, r \in \mathbb{Z}$ with $\gcd(q, r)=1$.
Hence it suffices to show that for $p, q, r \in \mathbb{Z}$ with $\gcd(q, r)=1$,
the sequence
\begin{equation*}
v_1=\left(\begin{array}{c} 1 \\ 0 \end{array}\right),
v_2=\left(\begin{array}{c} 0 \\ 1 \end{array}\right),
v_3=\left(\begin{array}{c} -1 \\ p \end{array}\right),
v_4=\left(\begin{array}{c} q \\ r \end{array}\right)
\end{equation*}
go around the origin exactly once in this order
and the associated toric surface $X(\Delta)$
is a toric log del Pezzo surface with $|\mathrm{Sing}(X(\Delta))|=2$ if and only if
$p \leq 1$ and $r \leq \min\{-pq-2, -2, -q-1, q-pq-1\}$.

If $v_1, \ldots, v_4$ determine a complete fan $\Delta$, then we have
\begin{itemize}
\item $\sigma_1$ and $\sigma_2$ are nonsingular;
\item $f(2) \geq 1 \Leftrightarrow p \leq 1$;
\item $\sigma_3$ is singular if and only if $r \leq -pq-2$;
\item $\sigma_4$ is singular if and only if $r \leq -2$;
\item $f(1) \geq 1 \Leftrightarrow r \leq -q$;
\item $f(3) \geq 1 \Leftrightarrow r \leq q-pq$.
\end{itemize}
Suppose that $X(\Delta)$
is a toric log del Pezzo surface with $|\mathrm{Sing}(X(\Delta))|=2$.
Since $r \leq -2$ and $v_4$ is primitive, we have $r \leq -q-1$.
If $r=q-pq$, then $q=\pm1$ since $v_4$ is primitive, which contradicts that $r \leq -pq-2$.
Hence $r \leq q-pq-1$.
Therefore $p \leq 1$ and $r \leq \min\{-pq-2, -2, -q-1, q-pq-1\}$.
Conversely, suppose that $p \leq 1$ and $r \leq \min\{-pq-2, -2, -q-1, q-pq-1\}$.
It suffices to show that $f(4) \geq 1$, that is, $p(1-q)-2r \geq 1$.
Since $r \leq -2$ and $v_4$ is primitive, we have $q \ne 0$.

{\it Case 1}. Suppose $q \geq 1$. Then $p \leq 1$ implies $p(1-q) \geq 1-q$.
Since $r \leq -q-1$, we have $-2r \geq 2q+2$.
Thus $f(4)=p(1-q)-2r \geq (1-q)+(2q+2)=q+3 \geq 4$.

{\it Case 2}. Suppose $q \leq -1$. It suffices to show $pq(1-q)-2qr \leq q$.
The assumption $pq \leq -r-2$ implies $pq(1-q) \leq (q-1)(r+2)$.
Since $r+2 \leq 0$ and $q-1 \geq 2q$, we have $(q-1)(r+2) \leq 2q(r+2)$.
Thus $pq(1-q)-2qr \leq 2q(r+2)-2qr=4q<q$.

In every case, we obtain $f(4) \geq 1$.
Therefore $X(\Delta)$ is a toric log del Pezzo surface with $|\mathrm{Sing}(X(\Delta))|=2$.

{\it The case where $d \geq 5$}.
First we show the following claims:

\begin{claim}\label{claim1}
If $X(\Delta)$ is a toric log del Pezzo surface
with $d \geq 4$ and $|\mathrm{Sing}(X(\Delta))|=2$,
then there is a sequence of toric log del Pezzo surfaces
\begin{equation*}
X(\Delta)=X_d\stackrel{\pi_d}{\longrightarrow}
X_{d-1}\stackrel{\pi_{d-1}}{\longrightarrow}\cdots\stackrel{\pi_6}{\longrightarrow}
X_5\stackrel{\pi_5}{\longrightarrow}X_4,
\end{equation*}
where $\pi_i$ is the blow-up of $X_{i-1}$ at a nonsingular $(\mathbb{C}^*)^2$-fixed point
for $5 \leq i \leq d$.
\end{claim}

\begin{proof}[Proof of Claim \ref{claim1}]
We use induction on $d$.
It is obvious for $d=4$. Assume $d \geq 5$.
We may assume that $\sigma_1$ and $\sigma_2$ are nonsingular.
If $X(\Delta)=X_d$ cannot be obtained by blowing-up a toric log del Pezzo surface
at a nonsingular $(\mathbb{C}^*)^2$-fixed point,
then $\sigma_3, \ldots, \sigma_d$ are all singular by Lemma \ref{RR},
which contradicts that $|\mathrm{Sing}(X(\Delta))|=2$.
Hence $v_2=v_1+v_3$.
Let $X_{d-1}$ be the toric log del Pezzo surface
associated to the fan formed from the sequence $v_1, v_3, v_4, \ldots, v_d$.
Then we have the blow-up $\pi_d:X_d \rightarrow X_{d-1}$.
By the hypothesis of induction,
there is a sequence of toric log del Pezzo surfaces
\begin{equation*}
X(\Delta)=X_d\stackrel{\pi_d}{\longrightarrow}
X_{d-1}\stackrel{\pi_{d-1}}{\longrightarrow}
X_{d-2}\stackrel{\pi_{d-2}}{\longrightarrow}\cdots\stackrel{\pi_6}{\longrightarrow}
X_5\stackrel{\pi_5}{\longrightarrow}X_4,
\end{equation*}
where $\pi_i$ is the blow-up of $X_{i-1}$ at a nonsingular $(\mathbb{C}^*)^2$-fixed point
for $5 \leq i \leq d$.
\end{proof}

\begin{claim}\label{claim2}
If $X$ is a toric log del Pezzo surface with $|\mathrm{Sing}(X)|=2$, then $\rho(X) \leq 3$.
\end{claim}

\begin{proof}[Proof of Claim \ref{claim2}]
Let $X'$ be a toric log del Pezzo surface of Picard number two with $|\mathrm{Sing}(X')|=2$,
and let $X$ be the blow-up of $X'$ at one nonsingular $(\mathbb{C}^*)^2$-fixed point.
Then $X$ is isomorphic to $X(\Delta)$,
where $\Delta$ is the complete fan in $\mathbb{R}^2$ formed from the sequence
\begin{equation*}
v_1=\left(\begin{array}{c} 1 \\ 0 \end{array}\right),
v_2=\left(\begin{array}{c} 0 \\ 1 \end{array}\right),
v_3=\left(\begin{array}{c} -1 \\ p+1 \end{array}\right),
v_4=\left(\begin{array}{c} -1 \\ p \end{array}\right),
v_5=\left(\begin{array}{c} q \\ r \end{array}\right)
\end{equation*}
for $p, q, r \in \mathbb{Z}$ with
$p \leq 1, r \leq \min\{-pq-2, -2, -q-1, q-pq-1\}, \gcd(q, r)=1$.
By Claim \ref{claim1}, it suffices to show that
if $X(\Delta)$ is log del Pezzo,
then the blow-up of $X(\Delta)$ at any nonsingular $(\mathbb{C}^*)^2$-fixed point is not log del Pezzo.
Suppose that $X(\Delta)$ is log del Pezzo.
Proposition \ref{intersection number} and $f(4)=q+1$ imply $q \geq 0$.
Since $r \leq -2$ and $v_5$ is primitive, we have $q \geq 1$.
The nonsingular $(\mathbb{C}^*)^2$-fixed points of $X(\Delta)$
are $\mathrm{orb}(\sigma_1), \mathrm{orb}(\sigma_2), \mathrm{orb}(\sigma_3)$.
However, we see that
\begin{align*}
\det(v_5, v_1)+\det(v_1, v_1+v_2)+\det(v_1+v_2, v_5)
&=1-q\leq0,\\
\det(v_2+v_3, v_3)+\det(v_3, v_4)+\det(v_4, v_2+v_3)
&=0,\\
\det(v_2, v_3)+\det(v_3, v_3+v_4)+\det(v_3+v_4, v_2)
&=0.
\end{align*}
Hence the blow-up of $X(\Delta)$ at any nonsingular $(\mathbb{C}^*)^2$-fixed point is not log del Pezzo.
This completes the proof of Claim \ref{claim2}.
\end{proof}

By Claims \ref{claim1} and \ref{claim2},
it suffices to show that for $p, q, r \in \mathbb{Z}$ with $\gcd(q, r)=1$,
the sequence
\begin{equation*}
v_1=\left(\begin{array}{c} 1 \\ 0 \end{array}\right),
v_2=\left(\begin{array}{c} 0 \\ 1 \end{array}\right),
v_3=\left(\begin{array}{c} -1 \\ p+1 \end{array}\right),
v_4=\left(\begin{array}{c} -1 \\ p \end{array}\right),
v_5=\left(\begin{array}{c} q \\ r \end{array}\right)
\end{equation*}
go around the origin exactly once in this order
and the associated toric surface $X(\Delta)$
is a toric log del Pezzo surface with $|\mathrm{Sing}(X(\Delta))|=2$
if and only if $p \leq 0$ and $1 \leq q \leq -r-1$.

If $v_1, \ldots, v_5$ determine a complete fan $\Delta$, then we have
\begin{itemize}
\item $\sigma_1, \sigma_2, \sigma_3$ are nonsingular;
\item $\sigma_5$ is singular if and only if $r \leq -2$;
\item $f(1) \geq 1$ if and only if $q \leq -r$;
\item $f(2) \geq 1$ if and only if $p \leq 0$;
\item $f(3)=1$;
\item $f(4) \geq 1$ if and only if $q \geq 0$.
\end{itemize}
If $X(\Delta)$ is a toric log del Pezzo surface with $|\mathrm{Sing}(X(\Delta))|=2$,
then $p \leq 0$ and $1 \leq q \leq -r-1$, since $r \leq -2$ and $v_5$ is primitive.
Conversely, suppose that $p \leq 0$ and $1 \leq q \leq -r-1$.
We need to show that $f(5) \geq 1$ and $\det(v_4, v_5) \geq 2$.
Since $p(1-q) \geq 0$ and $r \leq -2$,
we have $f(5)=p(1-q)-2r \geq 4$ and $\det(v_4, v_5)=-r-pq \geq 2$.
Thus $X(\Delta)$ is a toric log del Pezzo surface with $|\mathrm{Sing}(X(\Delta))|=2$.
This completes the proof of Theorem \ref{2singularities}.
\end{proof}

\section{Proof of Theorem \ref{3singularities}}

In this section, we prove Theorem \ref{3singularities}.
First we show the following lemma:

\begin{lemma}\label{SRSRS}
There are no toric log del Pezzo surfaces $X(\Delta)$ with $d=5$ and
$\mathrm{Sing}(X(\Delta))=\{\mathrm{orb}(\sigma_1), \mathrm{orb}(\sigma_3), \mathrm{orb}(\sigma_5)\}$.
\end{lemma}

\begin{proof}
Assume for contradiction that there exists a toric log del Pezzo surface $X(\Delta)$ with $d=5$ and
$\mathrm{Sing}(X(\Delta))=\{\mathrm{orb}(\sigma_1), \mathrm{orb}(\sigma_3), \mathrm{orb}(\sigma_5)\}$.
We may assume that $v_2=\left(\begin{array}{c} -1 \\ 0 \end{array}\right),
v_3=\left(\begin{array}{c}0 \\ -1 \end{array}\right)$.
Then $y_1 \geq 2, x_4 \geq 2, x_1+y_1 \geq 1, x_4+y_4 \geq 1$.

{\it Case 1}. Suppose $y_5 \leq 0$. Then $x_5 \geq 1$.
We show $x_5+y_5-x_4-y_4 \geq 1$.
Since $f(5)=(x_5-x_4)y_1-x_1(y_5-y_4)+1$ and $f(4)=x_4-x_5+1$,
Proposition \ref{intersection number}
gives $(x_5-x_4)y_1-x_1(y_5-y_4) \geq 0$ and $x_4 \geq x_5$.
It follows that $1=\det(v_4, v_5)=x_4y_5-x_5y_4 \leq x_5y_5-x_5y_4=x_5(y_5-y_4)$
and thus $y_5-y_4 \geq 1$.
Since $x_1 \geq 1-y_1$, we have $x_1(y_5-y_4) \geq (1-y_1)(y_5-y_4)$.
Hence
\begin{align*}
0 &\leq (x_5-x_4)y_1-x_1(y_5-y_4) \\
&\leq (x_5-x_4)y_1-(1-y_1)(y_5-y_4) \\
&=(x_5+y_5-x_4-y_4)y_1-(y_5-y_4) \\
&\leq (x_5+y_5-x_4-y_4)y_1-1.
\end{align*}
Therefore, $x_5+y_5-x_4-y_4 \geq 1$.

Let $v'=\left(\begin{array}{c} -1 \\ 1 \end{array}\right)$.
We consider the fan $\Delta'$ formed from the sequence $v', v_3, v_4, v_5$.
We calculate
\begin{align*}
\det(v_5, v')+\det(v', v_3)+\det(v_3, v_5)&=(x_5+y_5)+1+x_5 \\
&\geq (x_4+y_4+1)+1+x_5 \geq 3+x_5 \geq 4, \\
\det(v', v_3)+\det(v_3, v_4)+\det(v_4, v')&=1+x_4+(x_4+y_4) \geq 1+x_4+1 \geq 4, \\
\det(v_3, v_4)+\det(v_4, v_5)+\det(v_5, v_3)&=f(4) \geq 1, \\
\det(v_4, v_5)+\det(v_5, v')+\det(v', v_4)&=1+x_5+y_5-x_4-y_4 \geq 2.
\end{align*}
Hence $X(\Delta')$ is a toric log del Pezzo surface.
However, $\det(v_3, v_4) \ne 1$ and $\det(v_5, v')=x_5+y_5 \geq x_4+y_4+1 \geq 2$
while $\det(v', v_3)=\det(v_4, v_5)=1$,
which contradicts Theorem \ref{2singularities}.

{\it Case 2}. Suppose $y_5=1$.
Then $Q=\mathrm{conv}\{v_2, v_3, v_4, v_5\}$ is an LDP-polygon with four vertices
and the associated toric log del Pezzo surface $X_Q$ has exactly one singular point.
By Theorem \ref{Dais}, we have either $v_2=v_5+v_3$ or $v_5=v_4+v_2$.
If $v_2=v_5+v_3$, then
$v_5=v_2-v_3=\left(\begin{array}{c} -1 \\ 1 \end{array}\right)$
and thus $2 \leq \det(v_5, v_1)=-x_1-y_1 \leq -1$, a contradiction.
If $v_5=v_4+v_2$, then
$v_4=v_5-v_2=\left(\begin{array}{c} x_5+1 \\ 1 \end{array}\right)$
and thus $1 \leq \det(v_4, v_5)+\det(v_5, v_1)+\det(v_1, v_4)=1-y_1 \leq 1-2=-1$, a contradiction.

{\it Case 3}. Suppose $y_5 \geq 2$.
Then $Q=\mathrm{conv}\{v_2, v_3, v_4, v_5\}$ is an LDP-polygon with four vertices
and the associated toric log del Pezzo surface $X_Q$ has exactly two singular points.
However, $\det(v_3, v_4) \ne 1$ and $\det(v_5, v_2)=y_5 \geq 2$
while $\det(v_2, v_3)=\det(v_4, v_5)=1$,
which contradicts Theorem \ref{2singularities}.

Thus we have reached a contradiction in every case.
\end{proof}

\begin{proof}[Proof of Theorem \ref{3singularities}]
Let $X(\Delta)$ be a toric log del Pezzo surface with three singular points.
If $d \leq 4$, then there is nothing to prove.
Assume $d=5$.
By Lemma \ref{SRSRS},
there exists $i$ such that $\sigma_i$ and $\sigma_{i+1}$ are nonsingular.
We may assume that $i=1,
v_1=\left(\begin{array}{c} 1 \\ 0 \end{array}\right),
v_2=\left(\begin{array}{c} 0 \\ 1 \end{array}\right)$.
Then
\begin{equation*}
v_3=\left(\begin{array}{c} -1 \\ p \end{array}\right),
v_4=\left(\begin{array}{c} q \\ r \end{array}\right),
v_5=\left(\begin{array}{c} s \\ t \end{array}\right)
\end{equation*}
for some $p, q, r, s, t \in \mathbb{Z}$ with $\gcd(q, r)=\gcd(s, t)=1$.
Hence it suffices to show that for $p, q, r, s, t \in \mathbb{Z}$ with $\gcd(q, r)=\gcd(s, t)=1$,
the sequence
\begin{equation*}
v_1=\left(\begin{array}{c} 1 \\ 0 \end{array}\right),
v_2=\left(\begin{array}{c} 0 \\ 1 \end{array}\right),
v_3=\left(\begin{array}{c} -1 \\ p \end{array}\right),
v_4=\left(\begin{array}{c} q \\ r \end{array}\right),
v_5=\left(\begin{array}{c} s \\ t \end{array}\right)
\end{equation*}
go around the origin exactly once in this order
and the associated toric surface $X(\Delta)$
is a toric log del Pezzo surface with $|\mathrm{Sing}(X(\Delta))|=3$ if and only if
\begin{align*}
p & \leq 1,\\
r & \leq \min\{-1, -pq-2, q-pq-1, -pq+qt-rs+ps+t-1\},\\
t & \leq \min\{-2, -s-1, qt-rs+r-1\},\\
2 & \leq qt-rs.
\end{align*}

If $v_1, \ldots, v_5$ determine a complete fan $\Delta$, then we have
\begin{itemize}
\item $\sigma_1$ and $\sigma_2$ are nonsingular;
\item $f(2) \geq 1 \Leftrightarrow p \leq 1$;
\item $\sigma_3$ is singular if and only if $r \leq -pq-2$;
\item $f(3) \geq 1 \Leftrightarrow r \leq q-pq$;
\item $f(4) \geq 1 \Leftrightarrow r \leq -pq+qt-rs+ps+t-1$;
\item $\sigma_5$ is singular if and only if $t \leq -2$;
\item $f(1) \geq 1 \Leftrightarrow t \leq -s$;
\item $f(5) \geq 1 \Leftrightarrow t \leq qt-rs+r-1$;
\item $\sigma_4$ is singular if and only if $2 \leq qt-rs$.
\end{itemize}
Suppose that $X(\Delta)$ is a toric log del Pezzo surface with $|\mathrm{Sing}(X(\Delta))|=3$.
Since $t \leq -2$ and $v_5$ is primitive, we have $t \leq -s-1$.
If $r=q-pq$, then $q=\pm1$ since $v_4$ is primitive,
which contradicts that $r \leq -pq-2$. Hence $r \leq q-pq-1$.
It remains to show that $r \leq -1$.
If $p \leq 0$, then $r \leq -1$ holds obviously.
Assume $p=1$. Proposition \ref{intersection number} and $f(3)=1-r$ imply $r \leq 0$.
If $r=0$, then $q=-1$ since $v_4$ is primitive, which contradicts that $r \leq -pq-2$.
Hence $r \leq -1$.
Therefore we obtain the required inequalities.
The converse is obvious.

To prove the remaining part of the theorem, we need the following claim:

\begin{claim}\label{claim3}
If $X(\Delta)$ is a toric log del Pezzo surface
with $d \geq 5$ and $|\mathrm{Sing}(X(\Delta))|=3$,
then there is a sequence of toric log del Pezzo surfaces
\begin{equation*}
X(\Delta)=X_d\stackrel{\pi_d}{\longrightarrow}
X_{d-1}\stackrel{\pi_{d-1}}{\longrightarrow}\cdots\stackrel{\pi_7}{\longrightarrow}
X_6\stackrel{\pi_6}{\longrightarrow}X_5,
\end{equation*}
where $\pi_i$ is a blow-up of $X_{i-1}$ at a nonsingular $(\mathbb{C}^*)^2$-fixed point
for $6 \leq i \leq d$.
\end{claim}

\begin{proof}[Proof of Claim \ref{claim3}]
We use induction on $d$.
It is obvious for $d=5$.
Assume $d=6$.
Assume for contradiction that $X(\Delta)$ cannot be obtained by blowing-up a toric log del Pezzo surface
at a nonsingular $(\mathbb{C}^*)^2$-fixed point.
If there exists $i$ such that $\sigma_i$ and $\sigma_{i+1}$ are nonsingular,
then the remaining two-dimensional cones are all singular by Lemma \ref{RR},
which contradicts that $|\mathrm{Sing}(X(\Delta))|=3$.
Hence if $\sigma_i$ is nonsingular,
then both $\sigma_{i-1}$ and $\sigma_{i+1}$ are singular.
We may assume that
$\mathrm{Sing}(X(\Delta))=\{\mathrm{orb}(\sigma_1), \mathrm{orb}(\sigma_3), \mathrm{orb}(\sigma_5)\}$.
By Lemma \ref{SRS} (1), we have $\det(v_4, v_1) \geq 2$.
We consider the LDP-polygon $\mathrm{conv}\{v_1, v_2, v_3, v_4, v_5\}$.
If $\det(v_5, v_1)=1$,
then this contradicts Theorem \ref{2singularities}.
Otherwise this contradicts Lemma \ref{SRSRS}.
Hence $X(\Delta)$ is the blow-up of a toric log del Pezzo surface
at a nonsingular $(\mathbb{C}^*)^2$-fixed point.

Assume $d \geq 7$.
We may assume that $\sigma_1$ and $\sigma_2$ are nonsingular.
If $X(\Delta)=X_d$ cannot be obtained by blowing-up a toric log del Pezzo surface
at a nonsingular $(\mathbb{C}^*)^2$-fixed point,
then $\sigma_3, \ldots, \sigma_d$ are all singular by Lemma \ref{RR},
which contradicts that $|\mathrm{Sing}(X(\Delta))|=3$.
Hence $v_2=v_1+v_3$.
Let $X_{d-1}$ be the toric log del Pezzo surface
associated to the fan formed from the sequence $v_1, v_3, v_4, \ldots, v_d$.
Then we have the blow-up $\pi_d:X_d \rightarrow X_{d-1}$.
By the hypothesis of induction,
there is a sequence of toric log del Pezzo surfaces
\begin{equation*}
X(\Delta)=X_d\stackrel{\pi_d}{\longrightarrow}
X_{d-1}\stackrel{\pi_{d-1}}{\longrightarrow}
X_{d-2}\stackrel{\pi_{d-2}}{\longrightarrow}\cdots\stackrel{\pi_7}{\longrightarrow}
X_6\stackrel{\pi_6}{\longrightarrow}X_5,
\end{equation*}
where $\pi_i$ is a blow-up of $X_{i-1}$ at a nonsingular $(\mathbb{C}^*)^2$-fixed point
for $6 \leq i \leq d$.
\end{proof}

Let $X'$ be a toric log del Pezzo surface of Picard number three with $|\mathrm{Sing}(X')|=3$,
and let $X$ be the blow-up of $X'$ at one nonsingular $(\mathbb{C}^*)^2$-fixed point.
Then $X$ is isomorphic to $X(\Delta)$, where $\Delta$ is the complete fan in $\mathbb{R}^2$
formed from the sequence
\begin{align*}
&v_1=\left(\begin{array}{c} 1 \\ 0 \end{array}\right),
v_2=\left(\begin{array}{c} 0 \\ 1 \end{array}\right),
v_3=\left(\begin{array}{c} -1 \\ p+1 \end{array}\right),\\
&v_4=\left(\begin{array}{c} -1 \\ p \end{array}\right),
v_5=\left(\begin{array}{c} q \\ r \end{array}\right),
v_6=\left(\begin{array}{c} s \\ t \end{array}\right)
\end{align*}
for $p, q, r, s, t \in \mathbb{Z}$ with
\begin{align*}
p & \leq 1,\\
r & \leq \min\{-1, -pq-2, q-pq-1, -pq+qt-rs+ps+t-1\},\\
t & \leq \min\{-2, -s-1, qt-rs+r-1\},\\
2 & \leq qt-rs
\end{align*}
and $\gcd(q, r)=\gcd(s, t)=1$.
By Claim \ref{claim3}, to prove the remaining part
it suffices to show that if $X(\Delta)$ is log del Pezzo,
then the blow-up of $X(\Delta)$ at any nonsingular $(\mathbb{C}^*)^2$-fixed point is not log del Pezzo.
Suppose that $X(\Delta)$ is log del Pezzo.
Proposition \ref{intersection number} and $f(4)=q+1$ imply $q \geq 0$.
Since $rs \leq qt-2 \leq -2$, we have $s \geq 1$.
The nonsingular $(\mathbb{C}^*)^2$-fixed points of $X(\Delta)$
are $\mathrm{orb}(\sigma_1), \mathrm{orb}(\sigma_2), \mathrm{orb}(\sigma_3)$.
However, we see that
\begin{align*}
&\det(v_6, v_1)+\det(v_1, v_1+v_2)+\det(v_1+v_2, v_6)=1-s\leq0,\\
&\det(v_2+v_3, v_3)+\det(v_3, v_4)+\det(v_4, v_2+v_3)=0,\\
&\det(v_2, v_3)+\det(v_3, v_3+v_4)+\det(v_3+v_4, v_2)=0.
\end{align*}
Hence the blow-up of $X(\Delta)$ at any nonsingular $(\mathbb{C}^*)^2$-fixed point is not log del Pezzo.
This completes the proof of Theorem \ref{3singularities}.
\end{proof}

Finally, we give a proposition that holds even for $|\mathrm{Sing}(X(\Delta))| \geq 4$.

\begin{proposition}
Let $X(\Delta)$ be a toric log del Pezzo surface.
Then there exists an integer $n$ with $0 \leq n \leq d$ such that
$\sigma_i$ is nonsingular for $1 \leq i \leq n$
and $\sigma_i$ is singular for $n+1 \leq i \leq d$,
possibly after reordering the primitive generators $v_1, \ldots, v_d$.
\end{proposition}

\begin{proof}
We may assume
that $X(\Delta)$ cannot be obtained by blowing-up a toric log del Pezzo surface
at a nonsingular $(\mathbb{C}^*)^2$-fixed point.
If $d \leq 5$, then the assertion follows from Theorems \ref{2singularities} and \ref{3singularities}.
Assume $d=6$.
If $|\mathrm{Sing}(X(\Delta))| \geq 5$,
then the assertion is obvious.
If $|\mathrm{Sing}(X(\Delta))| \leq 3$,
then the assertion follows from Theorems \ref{2singularities} and \ref{3singularities}.
Hence we may further assume that $|\mathrm{Sing}(X(\Delta))|=4$
and $\sigma_2$ is nonsingular.
Assume for contradiction that the other two-dimensional nonsingular cone
is one of $\sigma_4, \sigma_5, \sigma_6$.

{\it Case 1}. Suppose that either $\sigma_4$ or $\sigma_6$ is nonsingular.
We may assume that $\sigma_4$ is nonsingular.
By Lemma \ref{SRS} (1), we have $\det(v_4, v_1) \geq 2$.
We consider the LDP-polygon $Q=\mathrm{conv}\{v_1, v_2, v_3, v_4, v_5\}$.
If $\det(v_5, v_1)=1$, then this contradicts Theorem \ref{2singularities}.
Otherwise this contradicts Theorem \ref{3singularities}.

{\it Case 2}. Suppose that $\sigma_5$ is nonsingular.
By Lemma \ref{SRS} (1),
the cones $\sigma_1, \sigma_2, \sigma_3$ cover more than a half-plane.
Similarly, $\sigma_4, \sigma_5, \sigma_6$ cover more than a half-plane.
This is a contradiction.

Hence the other two-dimensional nonsingular cone is either $\sigma_1$ or $\sigma_3$.
Therefore the assertion holds for $d=6$.

We prove the assertion for $d \geq 7$. We use induction on $d$.
If there exists $i$ such that $\sigma_i$ and $\sigma_{i+1}$ are nonsingular,
then the remaining two-dimensional cones are all singular by Lemma \ref{RR}.
Hence we may assume that
if $\sigma_i$ is nonsingular,
then both $\sigma_{i-1}$ and $\sigma_{i+1}$ are singular.
Assume for contradiction that there are $i$ and $j$ with $1 \leq i<j \leq d$ and $j \geq i+2$
such that $\sigma_i$ and $\sigma_j$ are nonsingular.
By Lemma \ref{SRS} (1), we must have $j=i+2$.
We may further assume that $i=2$. Then $\det(v_4, v_1) \geq 2$ by Lemma \ref{SRS} (1).
We consider the LDP-polygon $Q=\mathrm{conv}\{v_1, \ldots, v_{d-1}\}$.
We have $\det(v_1, v_2) \ne 1$ and $\det(v_3, v_4) \ne 1$ while $\det(v_2, v_3)=\det(v_4, v_5)=1$,
which contradicts the hypothesis of induction.
Hence $|\mathrm{Sing}(X(\Delta))|=d-1$.
Therefore the assertion holds for $d$.
This completes the proof.
\end{proof}

\end{document}